\documentclass[12pt]{amsart}
\usepackage{amsmath,amsfonts,amsthm,amscd,amssymb,mathrsfs,amssymb}
\usepackage{stmaryrd}
\usepackage{graphicx}
\newtheorem{theorem}{Theorem}

\newtheorem{proposition}[theorem]{Proposition}

\newtheorem{remark}[theorem]{Remark}

\newcommand{\vphi}{\varphi}
\numberwithin{equation}{section}
\numberwithin{theorem}{section}
\numberwithin{table}{section}
\numberwithin{table}{section}
\begin{document}
\bibliographystyle{amsalpha} 
\title[Sch\"odinger operators on compact manifolds]{The lowest eigenvalue of Schr\"odinger operators on compact manifolds}
\author{Michael G. Dabkowski}
\address{Dept. of Mathematics, Lawrence Technological University, Southfield, MI, 48075}
\email{mdabkowsk@ltu.edu}
\author{Michael T. Lock}
\address{Dept. of Mathematics, University of Texas, Austin, 
TX, 78712}
\email{mlock@math.utexas.edu}
\thanks{The second author was partially supported by NSF Grant DMS-1148490}
\date{May 15, 2016}
\begin{abstract}
The lowest eigenvalue of the Schr\"odinger operator $-\Delta+\mathcal{V}$ on a compact Riemannian manifold without boundary is studied.  We focus on the particularly subtle case of a sign changing potential with positive average.

\end{abstract}
\maketitle

\section{Introduction}
\label{Schrodinger_intro}

The time independent Schr\"odinger equation, ubiquitous in the study of quantum mechanics and partial differential equations, is given in the Euclidean context by
\begin{align}  
-\Delta_{Euc} \psi + V\psi = E \psi, \label{RnSchrodinger}
\end{align}
where $\Delta_{Euc}$ is the ordinary Euclidean Laplacian, $V$ is a function called the potential, and $E$ is a scalar representing the energy level. Equation \eqref{RnSchrodinger} is realized as an eigenvalue problem, where 
each eigenvalue corresponds to an energy level of a particle.  The lowest energy level {\em(ground state)} plays an essential role in the study of the~system.

This has a natural version in the Riemannian setting, and our interest is in the case that $(M,g)$ is a compact Riemannian manifold, which we always assume to be without boundary.  Indeed, we can consider the Schr\"odinger equation
\begin{align}  
\label{MSchrodinger}
-\Delta_g \varphi + \mathcal{V} \varphi = \mathcal{E} \varphi, 
\end{align}
where $\mathcal{V}$ is a smooth function on $M$, and our sign convention is such that the spectrum of $-\Delta_g$ is non-negative, see \eqref{laplacian}.  The spectrum of $-\Delta_g + \mathcal{V}$ is discrete and bounded from below.  Our interest here is in the sign of the lowest eigenvalue of this operator, in particular, the conditions under which the ground state has non-negative energy.

On a compact manifold, minimizing the Rayleigh quotient corresponding to this Schr\"odinger operator yields its lowest eigenvalue, see \eqref{rayleigh}.  Accordingly, it is elementary to determine the sign of the lowest eigenvalue in every situation except the case when the potential
\begin{align}
\begin{split}
\label{potential_condition}
&\text{(i)} \phantom{ii}\text{$\mathcal{V}$ changes sign on the manifold}\\
&\text{(ii)} \phantom{i}\int_M \mathcal{V}>0.
\end{split}
\end{align}
On a compact manifold, this case is quite subtle and will be the focus of our work.  

Interestingly, none of the existing results are applicable to the lowest eigenvalue problem in this setting.  For instance, the classical results in the non-compact setting which bound the number of negative eigenvalues less than a certain threshold from above, see \cite{Cwikel,Lieb_76,Rozenbljum_76,Reed-Simon,Lieb_80,Li-Yau,Rozenblum-Solomyak, Molchanov-Vainberg} and references therein, fail on compact manifolds as they rely upon estimates for the heat kernel which do not hold in this setting.
In fact, there are simple counter examples to these results in the compact setting.  Somewhat surprisingly, there is a dearth of work on eigenvalue problems for Schr\"odinger operators on compact manifolds.  Most notable, are the results \cite{Grigor'yan-Netrusov-Yau,Grigor'yan-Nadirashvili-Sire}, which provide a lower bound on the number of negative eigenvalues of $-\Delta_g+\mathcal{V}$.  However, while this bound is an incredible result, it is not sharp.  In particular, it provides no information when $\mathcal{V}$ satisfies conditions \eqref{potential_condition}, which is the focus of our work.  (Also, see \cite{Dolbeault-Esteban-Laptev-Loss} for an interesting result in the case of a positive potential.)
There has been no progress made in understanding if the lowest eigenvalue can ever be strictly positive in this realm of potentials.

Adopting the notation $L_{\mathcal{V}} = -\Delta_g + \mathcal{V}$, and assuming $\mathcal{V}$ satisfies \eqref{potential_condition}, we will be concerned with 
the lowest eigenvalue of this operator.  The ground state solution, or the eigenfunction corresponding to the lowest eigenvalue, is unique up to scale. Furthermore, since the manifold is compact, by the orthogonality of eigenfunctions, it is the only one with a sign and, without loss of generality, can be assumed to be positive.  Therefore, this is equivalent to the study of positive solutions
to \eqref{MSchrodinger}.

 
Given the rather minimal assumptions on $\mathcal{V}$, it is unlikely that an arbitrary non-negative $\lambda$ will be in the spectrum of $L_{\mathcal{V}}$, much less the lowest eigenvalue of this operator. 
Suppose, however, we were to fix a potential $\mathcal{V}$ while scaling the manifold, say to $(M, tg)$ for $t>0$.  Could it then be ensured, that there is some $t>0$ for which there exists a positive solution to $-\Delta_{tg}\vphi+\mathcal{V}\vphi=\lambda \vphi$?  Since $\Delta_{tg}= \frac{1}{t}\Delta_g$, this is equivalent to studying the equation
\begin{align}
\label{f(t)=t}
-\Delta_g \vphi+t\mathcal{V}\vphi=\lambda\vphi,
\end{align}
which is nothing more than the original problem with the potential scaled by $t$. While it is impermissible to scale the metric by $t=0$, it is valid to let $t=0$ in \eqref{f(t)=t}. 
This idea has a natural extension to a positive scaling of the potential $\mathcal{V}$ by a more general function of $t$. More precisely, we let $f:[0,\infty)\rightarrow [0,\infty)$ be any continuous function that satisfies the following three conditions
\begin{align}
\begin{split}
\label{f_conditions}
&\text{(i)} \phantom{iii}f(0)=0\\
&\text{(ii)} \phantom{ii}f(t)>0,\phantom{i}\text{for }t>0\\
&\text{(iii)} \phantom{i}\lim_{t\rightarrow\infty}f(t)=\infty.
\end{split}
\end{align}
It is important to note that using such an $f(t)$ as a multiplier on the potential preserves the conditions of \eqref{potential_condition} for all $t>0$.

Concisely, on a compact Riemannian manifold, of any real dimension, we consider the one-parameter family of Schr\"odinger operators given by
\begin{align}
\label{operators}
L_{f(t)\mathcal{V}}=-\Delta_g+f(t)\mathcal{V},
\end{align}
for $0\leq t<\infty$, where $f(t)$ and $\mathcal{V}$ respectively satisfy conditions \eqref{f_conditions} and \eqref{potential_condition}.  Recall that this is just a generalization of scaling the potential or the manifold and, for simplicity's sake, the reader can view $f(t)$ below as merely the parameter $t$.
For each value of $t$, this operator has discrete spectrum bounded from below, and each eigenvalue of $L_{f(t)\mathcal{V}}$ is a continuous function in the parameter $t$, see Section \ref{Schrodinger_background}.

It is a surprisingly subtle question of whether there exists a range over which the lowest eigenvalue is strictly positive.  This question is answered, in the affirmative, in Theorem \ref{t-schrodinger_thm_II}, and furthermore a lower bound on the size of this range is given.  It is frequently of interest, in both mathematics and physics, to understand when the ground state has zero energy.  While understanding the positive question is quite complicated, it is easy to tune the parameter to show that there is a range over which the lowest eigenvalue will be negative.  This is proved in Section \ref{zero_lambda} below.  Coupling these results together we are able to guarantee the existence of some $t^*\in(0,\infty)$ for which the lowest eigenvalue of $L_{f(t^*)\mathcal{V}}$ is zero.  This result, along with one concerning uniqueness, is provided in Theorem \ref{zero_eigenvalue}.
It is interesting to observe that these results hold for any dimension, as opposed to many classical results in the Euclidean setting for which complications arise in $2$-dimensions. 

\begin{theorem}
\label{t-schrodinger_thm_II}
On a compact manifold $(M,g)$, consider the one-parameter family of Schr\"odinger operators
\begin{align*}
L_{f(t)\mathcal{V}}=-\Delta_g+f(t)\mathcal{V},
\end{align*}
where $\mathcal{V}$ and $f(t)$ respectively satisfy conditions \eqref{potential_condition} and \eqref{f_conditions}.  Then, the lowest eigenvalue of $L_{f(t)\mathcal{V}}$
is strictly positive for all $t\in(0,\infty)$ for which
\begin{align*}
f(t)\leq\frac{\int_M\mathcal{V}}{P||\mathcal{V}||_{\infty}(4Vol(M)||\mathcal{V}||_{\infty}+\int_M\mathcal{V})},
\end{align*}
where $P>0$ is the Poincar\'e constant of the compact manifold.
\end{theorem}

This not only proves the existence of a positive lowest eigenvalue, but also an estimate on the size of the regime over which it is positive.  In particular, a very nice picture is presented when $f(t)=t$ where the lowest eigenvalue is guaranteed to be positive for $t\in \Big(0,\frac{\int_M\mathcal{V}}{P||\mathcal{V}||_{\infty}(4Vol(M)||\mathcal{V}||_{\infty}+\int_M\mathcal{V})}\Big]$.

Also, observe that Theorem \ref{t-schrodinger_thm_II} provides an estimate for the operator $-\Delta_g+\mathcal{V}$, where $\mathcal{V}$ satisfies  \eqref{potential_condition} is a fixed potential and the scaling parameter is absent.  Specifically, the lowest eigenvalue is positive if 
\begin{align}
\frac{P||\mathcal{V}||_{\infty}(4Vol(M)||\mathcal{V}||_{\infty}+\int_M \mathcal{V)}}{\int_M \mathcal{V}}\leq 1.
\end{align}

Given Theorem \ref{t-schrodinger_thm_II}, the question of existence for a zero lowest eigenvalue is reduced to showing that there exists a regime over which the lowest eigenvalue is negative.  We state the following result detailing the regimes over which the lowest eigenvalue is positive, zero and negative.


\begin{theorem}
\label{zero_eigenvalue}
On a compact manifold $(M,g)$, consider the one-parameter family of Schr\"odinger operators $L_{f(t)\mathcal{V}}$ as in Theorem \ref{t-schrodinger_thm_II}.
\begin{enumerate}
\item
\label{zero_existence} There exists a $t^*>0$ for which the lowest eigenvalue is zero. 
\item 
\label{part2}
There are non-empty open subsets, $I^+$ and $I^{-}$, of $t\in\big(0, \infty)$ upon which the lowest eigenvalues are respectively positive and negative.
\item
\label{part3}
If $f(t)$ is a strictly monotone increasing function, then there is a unique $t^*>0$ for which the lowest eigenvalue is zero, and
\begin{align*}
 I^{+}&=(0,t^*)\\
 I^{-}&=(t^*,\infty).
\end{align*}
\end{enumerate}
\end{theorem}

\begin{remark}
{\em
For any such family of operators, $L_{f(t)\mathcal{V}}$, there exists constants $C^+,C^->0$ so that $(0,C^+)\subset I^+$ and $(C^-,\infty)\subset I^-$.  In other words, for $t$ ``near enough'' to $0$ the lowest eigenvalue will always be positive, and for $t$ ``near enough'' to $\infty$ the lowest eigenvalue will always be negative.  Furthermore, note that in the case when $f$ is an increasing function, $I^{+}$ and $I^{-}$ are intervals constituted by values of $t<t^*$ and $t>t^*$ respectively.  Notice that Theorem \ref{t-schrodinger_thm_II} gives a lower bound on the interval $I^+$ in the case that $f(t)=t$.
}
\end{remark}

\subsection{Acknowledgements}
The authors would like to thank Joseph Conlon and Pablo Stinga for many useful conversations, as well as Herschel Viminah and Avram Mahnool for their insight into the physical aspects of this problem.


\section{Background and preliminaries}
\label{Schrodinger_background}
Let $(M,g)$ be a compact Riemannian manifold upon which we are investigating the Schr\"odinger equation 
\begin{align}
L_{f(t)\mathcal{V}}(\vphi)=-\Delta_g\vphi+f(t)\mathcal{V}\vphi=0,
\end{align}
where $f(t)$ and $\mathcal{V}$ satisfy \eqref{f_conditions} and \eqref{potential_condition} respectively.  Here $\Delta_g$ is the Laplace-Beltrami operator, which can be written locally as
\begin{align}
\label{laplacian}
\Delta_g=\frac{1}{\sqrt{\det(g)}}\partial_i\Big(\sqrt{\det(g)}g^{ij}\partial_j\Big).
\end{align}
Our sign convention is such that the spectrum of $-\Delta_g$ is non-negative.
For any fixed $t$, the potential $f(t) \mathcal{V}$ is continuous on $M$ and hence bounded. From this we conclude that the spectrum of $L_{f(t)\mathcal{V}}$ is discrete and bounded from below. List the eigenvalues in ascending order as functions of $t$, 
\begin{align}
\lambda_0(t)<\lambda_1(t)<\lambda_2(t)<\lambda_3(t)<\cdots,
\end{align}
and the corresponding eigenfunctions by $\vphi_{\lambda_i(t)}$, for $i\geq0$. It may be the case that a given eigenspace has dimension greater than one, in which case we will simplify select one element from the eigenspace. This will not cause us any problem as we are principally concerned with the lowest eigenfunction. 

The lowest eigenvalue of $L_{f(t)\mathcal{V}}$ is the key to unlocking the existence of a smooth positive solution to ${L_{f(t)\mathcal{V}}}(\vphi)=0$, as such a solution exists if and only if the operator has lowest eigenvalue
\begin{align}
\lambda_0(t)=0.
\end{align}
This is seen as follows. First, by a standard maximum principal argument, any eigenfunction that corresponds to the lowest eigenvalue will not change sign, see \cite{Kazdan}.  Now, if there were two independent eigenfunctions corresponding to the lowest eigenvalue, we could make the pair orthogonal which is an impossibility if neither change signs, so the eigenspace corresponding to the lowest eigenvalue is one dimensional.  Suppose now, that for some $i>0$, the eigenvalue $\lambda_i(t)=0$, so the lowest eigenvalue $\lambda_0(t)<0$. Indeed, since $L_{f(t)\mathcal{V}}$ is self-adjoint, 
\begin{align}
0=\langle L_{f(t)\mathcal{V}}(\vphi_{\lambda_i(t)}),\vphi_{\lambda_0(t)})\rangle=\langle\vphi_{\lambda_i(t)}, {L_{f(t)\mathcal{V}}}(\vphi_{\lambda_0(t)})\rangle=\lambda_0(t)\int_M\vphi_{\lambda_i(t)}\vphi_{\lambda_0(t)}.
\end{align}
Therefore, since $\vphi_{\lambda_0(t)}$ is the lowest eigenfunction, it does not change sign which forces $\vphi_{\lambda_i(t)}$ to change signs, so $ L_{f(t)\mathcal{V}} (\vphi)= 0$ will not have a positive solution. 

It is now evident that we wish to determine exactly when zero is the lowest eigenvalue of $L_{f(t)\mathcal{V}}$. To this end, we will employ the continuity of $\lambda_0(t)$ in $t.$ Recall that a second order differential operator $\mathcal{D} = a^{ij}(x)\partial_{x_i} \partial_{x_j} + b^l(x) \partial_{x_l} + c(x)$ is of Laplace type if $a^{ij} = g^{ij}$. Clearly, the operator $L_{f(t)\mathcal{V}}$ is of Laplace type.
Any smooth one-parameter family of self-adjoint Laplace type operators, such as $L_{f(t)\mathcal{V}}$, have $k^{th}$ eigenvalue $\lambda_k(t)$ with continuous dependence on $t$, see \cite{Park,Blavic-Bokan-Gilkey} and references therein. The continuity of the lowest eigenvalue in $t$ will be essential to our work here.

The lowest eigenvalue of the operator $L_{f(t)\mathcal{V}}$ is found by minimizing the Rayleigh quotient 
\begin{align}
\label{rayleigh}
\lambda_0(t) =\min_{0\not\equiv\vphi\in H^1}\frac{\int_M|\nabla\vphi|^2+f(t)\mathcal{V}\vphi^2}{\int_M \vphi^2}
\end{align}
exactly as in the Euclidean case, see \cite{Kazdan}.  We will see in Section \ref{Schrodinger_proof} that the lowest eigenvalue problem in the case that $\mathcal{V}$ changes signs and has positive integral is highly non-trivial.  However, when the potential does not satisfy these conditions, it is elementary to determine the sign of the lowest eigenvalue as we see in the following remark.

\begin{remark}{\em
Finding the sign of the lowest eigenvalue of $L_{\widetilde{\mathcal{V}}}=-\Delta+\widetilde{\mathcal{V}}$ is trivial is the case that the potential $\widetilde{\mathcal{V}}$ has a sign
or changes sign on the manifold and has non-positive average.  Since constant functions lie in the Sobolev space $H^1$ of a compact manifold, assuming $\widetilde{\mathcal{V}}\not\equiv 0$, the lowest eigenvalue is strictly less than the Rayleigh quotient evaluated at a constant.  Therefore, when $\int_M\widetilde{\mathcal{V}}\leq0$, the lowest eigenvalue is negative.  This includes both when $\widetilde{\mathcal{V}}$ is non-positive and not identically zero, and when it changes sign but has non-positive average.  In the case that $\widetilde{\mathcal{V}}$ is non-negative and not identically zero, it is clear that the Rayleigh quotient, and therefore the lowest eigenvalue, is strictly positive.  These are all of the possible cases of the potential except for the setting of our focus, when the potential satisfies \eqref{potential_condition}.
It is important to highlight the implicit use of the compactness of the manifold, which ensures that constant functions are in the Sobolev space $H^1$. 

} 

\end{remark}


\section{Proofs}
\label{Schrodinger_proof}

\subsection{Positive ground states}
\label{positive_lambda}
Here we prove Theorem \ref{t-schrodinger_thm_II}.
Any $\vphi\in H^1$ can be written as
\begin{align}
\begin{split}
\label{vphi_decomposition}
\vphi=&u+C_{\vphi},\phantom{=}\text{where}\\
\int_M u=0\phantom{=}&\text{and}\phantom{=}C_{\vphi}=\frac{1}{Vol(M)}\int_M\vphi.
\end{split}
\end{align}
Without loss of generality, we can assume that 
\begin{align}
\begin{split}
\label{f_wlog_assumption}
\int_M\vphi^2=&\int_M u^2+C_{\vphi}^2=1\\
\text{and}\phantom{========}\\
Vol(M)=&1,\\
\text{so}\phantom{========}\\
\int u^2=&1-C_{\vphi}^2.
\end{split}
\end{align}

Recall that the lowest eigenvalue of $L_{f(t)\mathcal{V}}$ is found by minimizing the Rayleigh quotient as in \eqref{rayleigh}.  From \eqref{vphi_decomposition} and \eqref{f_wlog_assumption}, the Rayleigh quotient becomes
\begin{align}
\label{f_rayleigh_positive_1}
\frac{\int_M|\nabla\vphi|^2+(f(t)\mathcal{V})\vphi^2}{\int_M \vphi^2}
=\int_M |\nabla u|^2+f(t)\mathcal{V}(u^2+2C_{\vphi}u+C_{\vphi}^2).
\end{align}
While the existence and uniqueness of a smooth minimizer, with unit $L^2$ norm, to the Rayleigh quotient is guaranteed, nothing is known about the actual function.  In particular, there are no qualities known that could help directly with the analysis of \eqref{f_rayleigh_positive_1}.  For instance, writing this minimizer as in \eqref{f_wlog_assumption}, nothing is known about the size of $C_{\vphi}^2$ versus $||u||_2^2$ within the unit bounds on each.  In turn, we must prove that there exists a range of $t>0$ for which the Rayleigh quotient \eqref{f_rayleigh_positive_1} is negative for any $0\leq C_{\vphi}\leq 1$ and all possible corresponding functions $u$.

Since $\int_M u=0$, it satisfies the Poincar\'e inequality
\begin{align}
\label{f_poincare}
\int_M u^2\leq P\int_M |\nabla u|^2,
\end{align}
where $P>0$ is the Poincar\'e constant which, on a compact manifold, is just the first nonzero eigenvalue of $-\Delta_g$.  Thus, by using \eqref{f_poincare}, the lower bound 
\begin{align}
\int_M |\nabla u|^2+(f(t)\mathcal{V})u^2\geq\int_M \Big(\frac{1}{P}+f(t)\mathcal{V}\Big)u^2
\end{align}
is obtained  on a component of the Rayleigh quotient \eqref{f_rayleigh_positive_1}.
Then, from \eqref{f_wlog_assumption} and an application of H\"older's inequality, we find that
\begin{align}
\begin{split}
\Big|\int_M (f(t)\mathcal{V})u\Big|\leq f(t)\int_M |\mathcal{V}u|&\leq f(t)||\mathcal{V}||_2||u||_2\\
&<f(t)||\mathcal{V}||_{\infty}||u||_2=f(t)||\mathcal{V}||_{\infty}\sqrt{1-C_{\vphi}^2},
\end{split}
\end{align}
from which the strict lower bound
\begin{align}
2C_{\vphi}f(t)\int_M\mathcal{V}\cdot u>-2f(t)||\mathcal{V}||_{\infty}C_{\vphi}\sqrt{1-C_{\vphi}^2}
\end{align}
on another component of the Rayleigh quotient \eqref{f_rayleigh_positive_1} is obtained.

Therefore, for any $0\leq C_{\vphi}\leq 1$, we find that \eqref{f_rayleigh_positive_1} satisfies the following sequence of inequalities:
\begin{align}
\begin{split}
\label{f_rayleigh_inequality}
\int_M |\nabla u|^2+&f(t)\mathcal{V}(u^2+2C_{\vphi}u+C_{\vphi}^2)\\
>& \int_M \Big(\frac{1}{P}+f(t)\mathcal{V}\Big)u^2-2f(t)||\mathcal{V}||_{\infty}C_{\vphi}\sqrt{1-C_{\vphi}^2}+C_{\vphi}^2f(t)\int_M\mathcal{V}\\
> & \Big(\frac{1}{P}-f(t)||\mathcal{V}||_{\infty}\Big)||u||_2^2-2f(t)||\mathcal{V}||_{\infty}C_{\vphi}\sqrt{1-C_{\vphi}^2}+C_{\vphi}^2f(t)\int_M \mathcal{V}\\
=&\Big(\frac{1}{P}-f(t)||\mathcal{V}||_{\infty}\Big)(1-C_{\vphi}^2)-2f(t)||\mathcal{V}||_{\infty}C_{\vphi}\sqrt{1-C_{\vphi}^2}+C_{\vphi}^2f(t)\int_M \mathcal{V}\\
=&\frac{1}{P}(1-C_{\vphi}^2)-f(t)||\mathcal{V}||_{\infty}\Big((1-C_{\vphi}^2)+2C_{\vphi}\sqrt{1-C_{\vphi}^2}\Big)+C_{\vphi}^2f(t)\int_M \mathcal{V}.
\end{split}
\end{align}
In order to prove this proposition, we will show that there exists some nonempty interval of $t$ for which the final expression in \eqref{f_rayleigh_inequality} is strictly positive.

Observe that the inequality
\begin{align}
\label{h_f_inequality_2}
\frac{1}{P}(1-C_{\vphi}^2)-f(t)||\mathcal{V}||_{\infty}\Big((1-C_{\vphi}^2)+2C_{\vphi}\sqrt{1-C_{\vphi}^2}\Big)\geq0
\end{align}
holds only whenever $t>0$ is such that the inequality
\begin{align}
\label{h_f_inequality_1}
f(t)\leq \frac{1}{P||\mathcal{V}||_{\infty}}\Big(1+\frac{2C_{\vphi}}{\sqrt{1-C_{\vphi}^2}}\Big)^{-1}
\end{align}
is satisfied.
Unfortunately, for any given $t>0$, there exists an $\epsilon>0$ so that the inequality \eqref{h_f_inequality_1}, and hence the inequality \eqref{h_f_inequality_2}, is violated for the range of constants $1-\epsilon<C_{\vphi}<1$.  Note though, that the inequality
\begin{align}
\label{h_f_inequality_6}
\frac{1}{P}(1-C_{\vphi}^2)-f(t)||\mathcal{V}||_{\infty}(1-C_{\vphi}^2)\geq0,
\end{align}
obtained by removing $-2f(t)||\mathcal{V}||_{\infty}C_{\vphi}\sqrt{1-C_{\vphi}^2}$ from \eqref{h_f_inequality_2},
holds for all $0\leq C_{\vphi}\leq1$ whenever $t$ is such that
\begin{align}
\label{h_f_inequality_7}
0\leq f(t)\leq\frac{1}{P\cdot||\mathcal{V}||_{\infty}},
\end{align}
with equality if and only if $C_{\vphi}=1$.  In a sense, what we will see, is that the term $-2f(t)||\mathcal{V}||_{\infty}C_{\vphi}\sqrt{1-C_{\vphi}^2}$ in \eqref{f_rayleigh_inequality} leads to quite a subtle difficulty.

In order to overcome this obstacle, and obtain the desired bound, we will show that a nonempty interval of $t>0$ exists so that, with respect to each $t$ in this interval, for $C_{\vphi}$ such that \eqref{h_f_inequality_1} is violated, the negativity of $-2f(t)||\mathcal{V}||_{\infty}C_{\vphi}\sqrt{1-C_{\vphi}^2}$ will be compensated for by the positivity of the $C_{\vphi}^2f(t)\int_M \mathcal{V}$ term that we have yet to utilize.  We will then show that, for this nonempty interval of $t$, the lowest eigenvalue is guaranteed to be positive.

To do this, we begin by finding for what range of $C_{\vphi}$ the inequality
\begin{align}
\label{h_f_inequality_3}
-2f(t)||\mathcal{V}||_{\infty}C_{\vphi}\sqrt{1-C_{\vphi}^2}+ C_{\vphi}^2f(t)\int_M  \mathcal{V}\geq 0
\end{align}
holds.  This range is precisely
\begin{align}
\label{h_f_inequality_4}
\sqrt{\frac{4}{\big(\frac{\int_M \mathcal{V}}{||\mathcal{V}||_{\infty}}\big)^2+4}}\leq C_{\vphi} \leq 1.
\end{align}
It is interesting to remark that this range is independent of $f(t)$, and the lower bound in \eqref{h_f_inequality_4} can be viewed as an invariant quantity of this entire family of Schr\"odinger operators.

Now, observe that there is a nonempty interval of $t>0$ such that inequality \eqref{h_f_inequality_1}, and hence inequality \eqref{h_f_inequality_2}, holds for
\begin{align}
\label{h_f_inequality_5}
0\leq C_{\vphi}\leq\sqrt{\frac{4}{\big(\frac{\int_M \mathcal{V}}{||\mathcal{V}||_{\infty}}\big)^2+4}},
\end{align}
for all $t$ in this interval.  This follows from the fact that $f(t)$ is continuous, $f(0)=0$, and $f(t)>0$ for $t>0$, which guarantees the existence of an interval of such $t>0$ near $t=0$.  Note, that there may be other $t$ away from this interval so that these conditions are satisfied as well, but such values of $t$ are only necessarily guaranteed to exist near $t=0$.  Also, note that \eqref{h_f_inequality_1} implies \eqref{h_f_inequality_7}, so \eqref{h_f_inequality_6} holds for all $t$ in this interval.

Finally, let $t^+>0$ be any element of this interval.  Then, since \eqref{h_f_inequality_1} and \eqref{h_f_inequality_2} are satisfied for $C_{\vphi}$ in the range \eqref{h_f_inequality_5}, from \eqref{h_f_inequality_3} and \eqref{h_f_inequality_4}, we see that the positivity of the term $C_{\vphi}^2f(t^+)\int_M \mathcal{V}$ compensates for the negativity assumed for the range of $C_{\vphi}$ where \eqref{h_f_inequality_1} is violated, and recall that in this range \eqref{h_f_inequality_6} holds.  Thus, for all $0\leq C_{\vphi}\leq 1$, we obtain the inequality
\begin{align}
\lambda_0(t^+)>\frac{1}{P}(1-C_{\vphi}^2)-f(t^+)||\mathcal{V}||_{\infty}\Big((1-C_{\vphi}^2)+2C_{\vphi}\sqrt{1-C_{\vphi}^2}\Big)+C_{\vphi}^2f(t^+)\int_M \mathcal{V}>0.
\end{align}

Since $f(t)$ is continuous and $f(0)=0$, certainly some component of the non-empty open subset of $t\in(0,\infty)$ for which the lowest eigenvalue is positive is contained near $t=0$.  When $f(t)$ is monotonically increasing, in particular, this set of $t$ is a connection open interval.  In other words, there exists some constant $T^+$ so that lowest eigenvalue of $L_{f(t)\mathcal{V}}$ is positive if and only if $t\in(0,T^+)$.  

The proof that $\lambda_0(t)>0$ if
\begin{align}
f(t)\leq\frac{\int_M\mathcal{V}}{P||\mathcal{V}||_{\infty}(4Vol(M)||\mathcal{V}||_{\infty}+\int_M\mathcal{V})}
\end{align}
follows almost immediately from the above once one removes the assumption that $Vol(M)=1$ and amends the proof accordingly.  This is because we want to extract a specific interval which is sensitive to rescalings whereas above we proved the existence of some nonempty interval.  Then, note that
\begin{align}
f(t)=\frac{\int_M\mathcal{V}}{P||\mathcal{V}||_{\infty}(4Vol(M)||\mathcal{V}||_{\infty}+\int_M\mathcal{V})}\end{align}
is the maximum value of $t$ for which the amended \eqref{h_f_inequality_1} holds for all $C_{\vphi}$ in the range the amended \eqref{h_f_inequality_5}.

\subsection{Zero ground states}
\label{zero_lambda}
Here we prove Theorem \ref{zero_eigenvalue}.
Recall, from Section \ref{Schrodinger_background}, that since
the operators $L_{f(t)\mathcal{V}}$ form a smooth one-parameter family of self-adjoint Laplace type operators, the $i^{th}$ eigenvalue of each, $\lambda_i(t)$, forms a continuous function of $t$.  In particular, the lowest eigenvalue $\lambda_0(t)$ is a continuous function of this parameter.  Therefore, to prove the existence of some $t^*>0$ so that $\lambda_0(t^*)=0$, it is only necessary to prove the existence of values $t^+,t^->0$ so that $\lambda_0(t^+)>0$ and $\lambda_0(t^-)<0$.  In Section \ref{positive_lambda}, we proved the existence of 
a positive lowest eigenvalue. 

We now give the existence result for an interval upon which the lowest eigenvalue of $L_{f(t)\mathcal{V}}$ is strictly negative.
\begin{proposition}
\label{f_-_prop}
There exists a nonempty interval of $t>0$ for which the lowest eigenvalue of the operator $L_{f(t)\cdot \mathcal{V}}$ is strictly negative.
\end{proposition}
\begin{proof}

Since $\mathcal{V}$ changes sign, there is a subset of $M$ on which $\mathcal{V}$ is strictly negative.  Specifically, this subset $M^-\subset M$ is defined as
\begin{align}
M^-:=\{x\in M: \mathcal{V}(x)<0\}.
\end{align}
Now, choose any $0\not\equiv\vphi\in C^{\infty}_0(M^-)$, and consider the associated signed values
\begin{align}
\begin{split}
C_1(\vphi)&=\int_{M}|\nabla \vphi|^2=\int_{M^-}|\nabla \vphi|^2>0\\
C_2(\vphi)&=\int_{M}\mathcal{V} \cdot\vphi^2=\int_{M^-}\mathcal{V} \cdot\vphi^2<0.
\end{split}
\end{align}
Clearly, both $C_1(\vphi)$ and $C_2(\vphi)$ are finite since $\vphi\in C^{\infty}_0(M^-)$ and $M$ is compact.
Therefore, since $\lim_{t\rightarrow 1}f(t)=\infty$, there exists some $t^->0$ so that
\begin{align}
C_1(\vphi)+f(t^-)C_2(\vphi)<0.
\end{align}
and, because the lowest eigenvalue is the minimizer of the Rayleigh quotient \eqref{rayleigh}, we see that $\lambda_0(t^-)<0$.  
\end{proof}

Part \eqref{zero_existence} and part \eqref{part2} of Theorem \ref{zero_eigenvalue} follow from Theorem \ref{t-schrodinger_thm_II} and
Proposition \ref{f_-_prop} by using continuity for the lowest eigenvalues of the one-parameter family of operators $L_{f(t)\mathcal{V}}$.  Lastly, we prove part \eqref{part3}, the uniqueness statements.

\begin{proposition}
\label{f_uniqueness}
Let $f(t)$ be a strictly monotone increasing function.  Then, there is a unique $t^*>0$ for which the equation
\begin{align*}
L_{f(t)\cdot \mathcal{V}}(\phi)=0
\end{align*}
has a smooth positive solution which is itself unique.
\end{proposition}
\begin{proof}
In the proof of Theorem \ref{zero_eigenvalue} part \eqref{zero_existence}, it was shown that there exists at least one $t^*>0$ with $\lambda_0(t^*)=0$.  It will now be shown that, given the strict monotonicity condition on $f(t)$, there is exactly one such $t^*$.

This is seen as follows.  Let $t^*>0$ be such that $\lambda_0(t^*)=0$ and denote the associated positive eigenfunction with unit $L^2$ norm by $\vphi_{\lambda_0(t^*)}$.  Then
\begin{align}
\int_M|\nabla\vphi_{\lambda_0(t^*)}|^2+(f(t^*)\cdot\mathcal{V})\vphi_{\lambda_0(t^*)}^2=0,
\end{align}
so $\int_M(f(t^*)\cdot\mathcal{V})\vphi_{\lambda_0(t^*)}^2<0$ since $\int_M|\nabla\vphi_{\lambda_0(t^*)}|^2>0$ because $\vphi_{\lambda_0(t^*)}$ is non-constant given that $t^*>0$.  

Thus, for any $t'>t^*$, the inequality
\begin{align}
\lambda_0(t')=\int_M|\nabla \vphi_{\lambda_0(t')}|^2+(f(t')\cdot\mathcal{V})\vphi_{\lambda_0(t')}<\int_M|\nabla \vphi_{\lambda_0(t^*)}|^2+(f(t')\cdot\mathcal{V})\vphi_{\lambda_0(t^*)}^2<0
\end{align}
holds, so beyond $t^*$ the lowest eigenvalue remains strictly negative from which we see that there is a unique $t^*$ with $\lambda_0(t^*)=0$.

\end{proof}

\bibliography{sKlsc_references}

\end{document}